\documentclass[11pt,reqno]{amsart}
\usepackage{amssymb,mathtools,calc,verbatim,enumitem,tikz,url,mathrsfs,cite,fullpage}
\usepackage{textcomp}
\usepackage{setspace}
\usepackage{amsthm}
\usepackage{amsmath}
\usepackage{graphicx}
\usepackage{marvosym}
\usepackage{empheq}
\usepackage{latexsym}
\usepackage[T1]{fontenc}
\usepackage{lmodern}
\usepackage{color}
\usepackage{hyperref}
\usepackage{cleveref}
\usepackage{refcount}
\usepackage{microtype}
\usepackage{placeins}

\usepackage{todonotes}

\newenvironment{poc}{\begin{proof}[Proof of Claim~\theclaim]}{\end{proof}}

\newtheorem{theorem}{Theorem}[section]
\newtheorem{lemma}[theorem]{Lemma}
\newtheorem{corollary}[theorem]{Corollary}

\newtheorem{proposition}[theorem]{Proposition}
\newtheorem{claim}{Claim}
\newcommand{\claimparent}{0}
\renewcommand{\theclaim}{\claimparent.\arabic{claim}}

\newcommand{\claimsof}[1]{%
  \setcounter{claim}{0}%
  \edef\claimparent{\getrefnumber{#1}}%
}
\newtheorem*{lemma*}{Lemma}

\theoremstyle{definition}

\newtheorem*{qu*}{Question}
\theoremstyle{remark}

\renewcommand\Pr{\operatorname{\mathbb{P}}}

\renewcommand\le{\leqslant}
\renewcommand\ge{\geqslant}
\renewcommand\to{\rightarrow}

\def\<{\langle}
\def\>{\rangle}

\newcommand{\iso}{\mathrm{iso}}

\title{An extremal theorem for non-isomorphic spanning trees}

\author{Zhifei Yan}
\address{ECOPRO, Institute for Basic Science, 55 Expo-ro, Yuseong-gu,
Daejeon, 34126, Korea}
\email{zhifeiyan@ibs.re.kr}

\author{Lu-Ming Zhang}
\address{Department of Mathematics, London School of Economics and Political Science,
Houghton Street, London, WC2A 2AE, United Kingdom}
\email{l.zhang100@lse.ac.uk}

\date{}

\newcommand{\tauiso}{\tau_{\iso}}

\usetikzlibrary{arrows.meta,positioning,calc}
\hypersetup{
  pdftitle={An extremal theorem for non-isomorphic spanning trees},
  pdfauthor={Zhifei Yan and Lu-Ming Zhang}
}
\begin{document}

\begin{abstract}
For a graph $G$, let $\tauiso(G)$ denote the number of isomorphism classes
of its spanning trees. For every fixed $d\ge3$ and all sufficiently large
$n$, we prove that every connected $n$-vertex graph $G$ with
$\delta(G)\ge d$ satisfies
\[
  \tauiso(G)\ge \tauiso(K_{d,n-d})=A_dn^{d-1}+O_d(n^{d-2}),
\]
for an explicit constant
$A_d>0$, and $K_{d,n-d}$ is the unique minimizer. This confirms a conjecture of Bitonti, Michel and Scott and
extends it to every $d\ge3$. We also show that any such graph with \(O(n^{d-1})\) spanning-tree types has all but a bounded number of vertices with the same \(d\) neighbours.
\end{abstract}

\maketitle

\section{Introduction}

Counting spanning trees is a classical problem in graph theory.
The Matrix--Tree Theorem has its origins in Kirchhoff's 1847
work~\cite{kirchhoff}. In its modern formulation, it expresses the
number $\tau(G)$ of spanning trees of a connected graph as a cofactor
of its Laplacian matrix, while Cayley's formula~\cite{cayley}
gives $\tau(K_n)=n^{n-2}$. Degree conditions also force many spanning
trees.
For regular graphs, McKay~\cite{mckay} obtained sharp upper bounds on the number of spanning trees, while Alon~\cite{alon} proved lower bounds. Kostochka~\cite{kostochka} extended Alon’s lower bounds to graphs with general degree sequences.
These lower bounds show that sufficiently large
minimum degree forces exponentially many spanning trees, but they count
different edge sets even when the resulting trees are isomorphic.

We instead count spanning trees up to isomorphism. Write $\tauiso(G)$
for the number of their isomorphism classes, viewing each tree as an
abstract unrooted graph. The distinction matters: even in $K_n$, a tree
$T$ has $n!/|\operatorname{Aut}(T)|$ labelled copies, so types occur with
different multiplicities. For the complete graph, all unlabelled trees
occur, and Otter~\cite{otter} proved in 1948 that
\[
  \tauiso(K_n)=(1+o(1))C\alpha^n n^{-5/2},
  \qquad C\approx0.535,\quad \alpha\approx2.956.
\]
For a general host graph, a large labelled count does not by itself
ensure a comparable number of types.

The study of tree types in a fixed graph goes back at least to
Zelinka~\cite{zelinka71}, who characterized the finite connected graphs
whose spanning trees are all isomorphic, answering a question of
Sedl\'a\v{c}ek. Vestergaard~\cite{vestergaard} later gave further proofs
and studied the corresponding question for finite and infinite graphs.
In 1978, Zelinka~\cite{zelinka78} proved that a connected graph containing
$k\ge2$ vertex-disjoint cycles has at least $k+1$ non-isomorphic spanning
trees, confirming a conjecture of Hartnell. These results link the
number of types to the structure of the graph. A natural next question
is how many types can be forced by a minimum-degree condition alone.

Recent progress has used random spanning trees. Let $\mathcal T$ be a
uniformly random spanning tree of $G$. A bound
$\Pr(\mathcal T\cong T)\le p$ for every tree $T$ implies
$\tauiso(G)\ge p^{-1}$. Lee~\cite{lee} proved that there are absolute
constants $\varepsilon,c>0$ such that, if $G$ is connected and all its
degrees lie in $[(1-\varepsilon)D,(1+\varepsilon)D]$ for sufficiently
large $D$, then
\[
  \Pr(\mathcal T\cong T)\le e^{-cn}
  \qquad\text{for every tree }T.
\]
Thus these almost regular graphs have exponentially many spanning-tree
types. Wang~\cite{wang} obtained a growth
rate arbitrarily close to Otter's constant: for every
$0<\varepsilon<\alpha$, regular graphs of sufficiently large degree
satisfying a suitable fixed spectral condition have at least
$(\alpha-\varepsilon)^n$ types for large $n$. Regularity and expansion
thus permit much stronger conclusions than a minimum-degree assumption.

Indeed, the complete bipartite graph $K_{d,n-d}$ has exponentially many
labelled spanning trees but only polynomially many types when $d$ is
fixed. The classical bipartite formula~\cite{moon} gives
\[
  \tau(K_{d,n-d})=d^{n-d-1}(n-d)^{d-1},
\]
whereas $\tauiso(K_{d,n-d})=\Theta_d(n^{d-1})$. Deleting the leaves on
the large side leaves a core of bounded order; for each core, only the
$d$ leaf counts on the small side remain, and their sum is fixed.
We determine the leading coefficient in
Proposition~\ref{prop:coefficient}; see Johnson and Nochumson~\cite{johnson}
for partition-based bounds for general complete bipartite graphs.
Thus the largest possible universal power of $n$ is $d-1$.

Bitonti, Michel and Scott~\cite{bitonti} proved that, for sufficiently
large $d$ and $n$ sufficiently large in terms of $d$, every connected
$n$-vertex graph of minimum degree at least $d$ satisfies
$\Pr(\mathcal T\cong T)\le n^{-\Omega(d)}$ for every tree $T$.
This resolves Lee's polynomial anti-concentration conjecture and gives
$$\tauiso(G)\ge n^{\Omega(d)}.$$
They further conjectured the optimal
counting bound $\Omega_d(n^{d-1})$ for all sufficiently large $d$
\cite[Conjecture~4.2]{bitonti}. We prove this for every $d\ge3$ and,
more precisely, identify the unique extremal graph for large $n$.

\begin{theorem}\label{thm:main}
For every fixed integer $d\ge3$ and all sufficiently large $n$, every
connected $n$-vertex graph $G$ with $\delta(G)\ge d$ satisfies
\[
  \tauiso(G)\ge\tauiso(K_{d,n-d}),
\]
with equality if and only if $G\cong K_{d,n-d}$.
\end{theorem}

The minimum also has an explicit leading coefficient. Writing $S(a,b)$
for a Stirling number of the second kind, set
\[
  A_d:=\frac{1}{d!(d-1)!}
       \sum_{k=1}^{d-1}d^{k-1}S(d-1,k).
\]
Proposition~\ref{prop:coefficient} gives
$$\tauiso(K_{d,n-d})=A_dn^{d-1}+O_d(n^{d-2});$$
for example,
$A_3=1/3$~\cite{mohr} and $A_4=29/144$~\cite{vandenboomen}. Decreasing the constant to cover the finitely
many smaller orders gives $\tauiso(G)\ge c_dn^{d-1}$ for every $n$.
The restriction $d\ge3$ is necessary, since every spanning tree of a
cycle is a path.

The main step is a stability theorem. It applies not just near the
minimum, but whenever the count is at most any fixed multiple of
$n^{d-1}$. Throughout the paper, graphs are finite and simple, and
$N_G(v)$ denotes the open neighbourhood of $v$.

\begin{theorem}[Stability]\label{thm:stability}
Fix $d\ge3$ and $C>0$. For all sufficiently large $n$, every connected
$n$-vertex graph $G$ with $\delta(G)\ge d$ and
$\tauiso(G)\le Cn^{d-1}$ has a set $S$ of $d$ vertices such that
\[
  \bigl|\{v:N_G(v)=S\}\bigr|=n-O_{d,C}(1).
\]
\end{theorem}

Thus all but a bounded number of vertices have the same $d$ neighbours.
A nonempty remainder, or an extra edge among these neighbours, then
forces strictly more tree types than $K_{d,n-d}$ for large $n$.
The proof counts types directly, using local changes in a DFS tree
that remain distinguishable after vertex labels are forgotten.

\medskip
\noindent\textbf{Organization.}
Section~\ref{sec:overview} gives the proof idea.
Section~\ref{sec:tools} develops the three counting tools used later.
Section~\ref{sec:stability} proves the stability theorem, and
Section~\ref{sec:extremality} compares the remaining bounded-remainder
graphs. Section~\ref{sec:mainproofs} then completes the proof of the main theorem.

\medskip
\noindent\textbf{Terminology used in this paper.}
For a $d$-set $S$, the \emph{common-neighbourhood class}
$X(S)=\{v:N_G(v)=S\}$ consists of the vertices whose neighbourhood is
exactly $S$. We often call the vertices of $S$ the \emph{hubs}; the vertices
outside $S\cup X(S)$ form the \emph{remainder}. Thus a large class with a
small remainder is the structure we expect near $K_{d,n-d}$.

A \emph{piece} is a connected set of at least two vertices on which we make
local changes while keeping the rest of the tree fixed; later it is usually a
descendant subtree of a DFS tree. If all outside neighbours of a piece $U$
lie in a set $S_U$, an $S_U$-rooted forest splits $U$ into \emph{branches},
each attached to one vertex of $S_U$. The \emph{branch profile} records the
rooted types of these branches and their attachment vertices, with
multiplicities. A \emph{core} appears only in the final extremal count: for a
spanning tree of $K_{d,n-d}$, delete the leaves in the large bipartition class
while keeping the $d$ hubs; the bounded tree that remains is the core. Heavy
and light subtrees are defined when they first appear in
Section~\ref{sec:stability}.

\medskip

\section{Proof overview}\label{sec:overview}

The complete bipartite graph explains the exponent $d-1$. After a bounded
core is fixed, almost every vertex on the large side is a leaf. A tree type
then records essentially only the $d$ leaf counts at the vertices on the
small side, and these counts have fixed sum. Thus there are only $d-1$
free parameters.

Our proof shows that every genuinely nontrivial piece creates one more
degree of freedom. Suppose that $U$ is a connected set with at least two
vertices and all its outside neighbours lie in a set $S$. Replacing the
tree inside $U$ by an $S$-rooted forest gives a set of branch profiles.
Lemma~\ref{lem:multiroot} shows that this set has affine dimension at least
$d$. By contrast, one leaf with $d$ possible attachment vertices has only
$d$ choices and affine dimension $d-1$. Lemma~\ref{lem:affine} then turns
many such pieces into order $t^d$ different profile sums. This extra
power is the source of the stability theorem.

We now sketch how the pieces are found. Fix $d,C$ and assume
\[
  \tauiso(G)\le Cn^{d-1}.
\]
For a $d$-set $S$, write
  $X(S):=\{v:N_G(v)=S\}$,$m=|X(S)|$ and $r=n-d-m$.
Proposition~\ref{prop:amplify} gives
\[
  \tauiso(G)\ge b_d\frac{m^d(r+1)}n.
\]
Hence a class with $m\ge n/2$ already has $r=O_{d,C}(1)$, while every
smaller class has size $O_{d,C}(n^{1-1/d})$. Thus, to prove stability, it
is enough to find one class containing at least half the vertices.

Assume that no such class exists. Proposition~\ref{prop:path} shows that
every DFS tree has height
\[
  H=O_{d,C}((\log n)^4).
\]
The small-class bound forces the leaves of the DFS tree to be spread over
many parents. A heavy-path decomposition then gives
\[
  N=\frac{n}{(\log n)^{O_d(1)}}
\]
disjoint light subtrees, all with the same order $b\ge2$. Their heavy
siblings stay fixed and make the sum of the local profiles visible in the
whole tree. Lemma~\ref{lem:recovery} formalizes this recovery. Finally,
Lemma~\ref{lem:guarded} chooses the root so that forgetting it costs
only a polylogarithmic factor. The result is
\[
  \tauiso(G)\ge \frac{n^d}{(\log n)^{O_d(1)}},
\]
contrary to $\tauiso(G)\le Cn^{d-1}$. Therefore a class of size at least
$n/2$ exists, and Proposition~\ref{prop:amplify} then gives the bounded
remainder required by Theorem~\ref{thm:stability}.

After stability, the problem is finite apart from the leaf counts.
Section~\ref{sec:extremality} first counts the bounded cores of
$K_{d,n-d}$. It then shows that any nonempty remainder gives a second
family of the same leading order, while an extra edge among the $d$ common
neighbours gives a new family of cores. Thus every other bounded-remainder
graph has strictly more types for large $n$. Applying this comparison to a
minimizer proves Theorem~\ref{thm:main}.

Figure~\ref{fig:proof} records this chain of ideas.

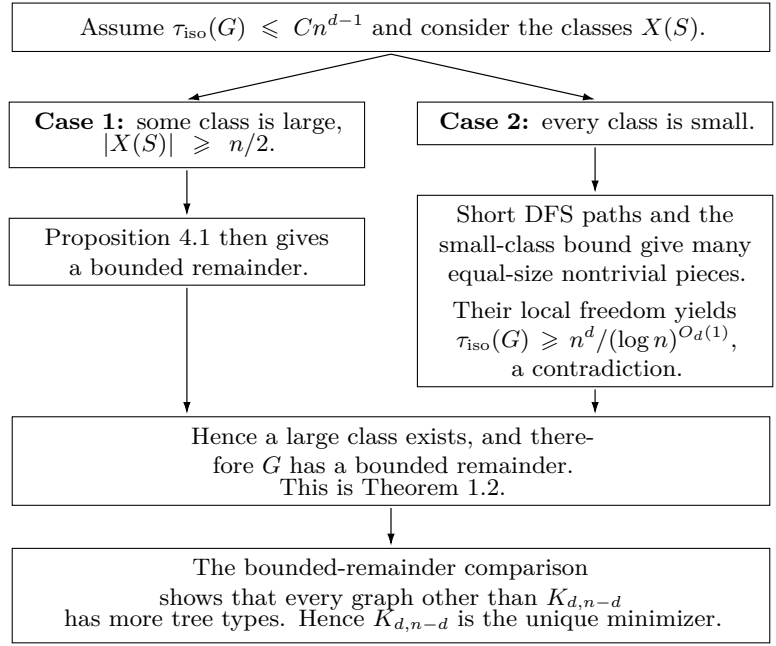
\begin{figure}[!htbp]
\centering
\begingroup
\setstretch{1}
\begin{tikzpicture}[
  casebox/.style={draw,line width=.4pt,align=center,
    text width=4.35cm,inner xsep=5pt,inner ysep=4pt,
    font=\footnotesize,outer sep=1pt},
  widebox/.style={draw,line width=.4pt,align=center,
    text width=9.6cm,inner xsep=6pt,inner ysep=4pt,
    font=\footnotesize,outer sep=1pt},
  arrow/.style={-{Latex[length=1.6mm,width=1mm]},line width=.45pt}
]
\node[widebox] (start) {Assume $\tauiso(G)\le Cn^{d-1}$ and consider the classes $X(S)$.};

\node[casebox,below=6mm of start,xshift=-2.7cm] (large) {
  \textbf{Case 1:} some class is large,\\[-1mm]
  $|X(S)|\ge n/2$.};
\node[casebox,below=6mm of start,xshift=2.7cm] (small) {
  \textbf{Case 2:} every class is small.};

\node[casebox,below=6mm of large] (large2) {
  Proposition~\ref{prop:amplify} then gives a bounded remainder.};
\node[casebox,below=6mm of small] (small2) {
  Short DFS paths and the small-class bound give many equal-size nontrivial pieces.\\[1mm]
  Their local freedom yields $\tauiso(G)\ge n^d/(\log n)^{O_d(1)}$, a contradiction.};

\node[widebox,below=10mm of $(large2.south)!0.5!(small2.south)$] (stable) {
  Hence a large class exists, and therefore $G$ has a bounded remainder.\\[-1mm]
  This is Theorem~\ref{thm:stability}.};

\node[widebox,below=5mm of stable] (main) {
  The bounded-remainder comparison shows that every graph other than $K_{d,n-d}$\\[-1mm]
  has more tree types. Hence $K_{d,n-d}$ is the unique minimizer.};

\draw[arrow] (start.south) -- (large.north);
\draw[arrow] (start.south) -- (small.north);
\draw[arrow] (large.south) -- (large2.north);
\draw[arrow] (small.south) -- (small2.north);
\coordinate (mergeL) at (large2.south |- stable.north);
\coordinate (mergeR) at (small2.south |- stable.north);
\draw[arrow] (large2.south) -- (mergeL);
\draw[arrow] (small2.south) -- (mergeR);
\draw[arrow] (stable.south) -- (main.north);
\end{tikzpicture}
\endgroup
\caption{Roadmap of the proof.}
\label{fig:proof}
\end{figure}

\section{Counting tools}\label{sec:tools}

The stability proof uses three tools from this section. First,
Lemma~\ref{lem:multiroot} shows that every nontrivial piece has a set of
branch profiles of affine dimension at least $d$. Second,
Lemma~\ref{lem:affine} shows that many such local choices still give many
different sums, even when the pieces themselves cannot be distinguished.
Third, Proposition~\ref{prop:path} shows that a long path already creates
many spanning-tree types; under the hypothesis $\tauiso(G)\le Cn^{d-1}$,
it implies that every DFS tree has polylogarithmic height.

We develop these tools in that order. The first two are local counting
statements; the third is the global path bound used throughout the stability
argument.

\subsection{Rooted trees and forest profiles}

We first describe the choices inside one piece. Let $C$ be a connected
graph and $S$ a disjoint set of labelled roots; ignore edges within $S$.
An \emph{$S$-rooted spanning forest} has exactly one vertex of $S$ in
each component, with unused roots allowed to be isolated. Its
\emph{branch profile} records, for each pair $(s,R)$, the number of
branches at $s$ of rooted tree type $R$, rooted at the neighbour of $s$.
For a single root, the profile determines the rooted tree. All affine
dimensions in this section are over $\mathbb R$.

\begin{lemma}\label{lem:multiroot}
Suppose $|C|\ge2$, at least one edge joins $C$ to $S$, and each vertex of
$C$ has degree at least $d\ge3$ in $C\cup S$. The profiles of the
$S$-rooted spanning forests have affine dimension at least $d$.
\end{lemma}

We start with a count for rooted trees. Each non-tree edge of a DFS
tree can be used to decrease the sum of distances from the root.
This produces different rooted types. We will use these types to
build profiles with one branch, then add profiles with several branches
to obtain the required dimension. Write $q(F,r)$ for the number of
rooted spanning-tree types of $(F,r)$.

\begin{lemma}\label{lem:surplus}
Every connected rooted graph $(F,r)$ satisfies
\[
 q(F,r)\ge e(F)-v(F)+2.
\]
\end{lemma}
\begin{proof}
Take a DFS tree $T$ rooted at $r$. Process its nonroot vertices in decreasing
original depth. At a vertex $y$, replace its parent successively by its
non-tree ancestor neighbours, also in decreasing original depth.

Every step preserves a tree, since every parent remains a strict ancestor
in the original tree. When $y$ is processed, none of its original ancestors
has been processed, so their paths to $r$ are unchanged. Each new parent of
$y$ is strictly closer to $r$ than the preceding one. The replacement
therefore strictly decreases the sum of distances from $r$.

Each non-tree edge is used exactly once, at its lower endpoint. There are
$e(F)-v(F)+1$ replacements. The initial tree and all the intermediate trees
have different distance sums, so they have distinct rooted types.
\end{proof}

A connected graph $F$ containing a cycle has $e(F)\ge v(F)$, so the lemma gives the
following immediate consequence. We will use it both for local forest
choices and for branches attached to a DFS path.

\begin{corollary}\label{cor:cyclic}
A connected rooted graph with a cycle has at least two rooted
spanning-tree types.
\end{corollary}

The rooted count gives profiles with a single branch. We also need
profiles that split the vertices into smaller branches. Two observations
show that one or two additional profiles will be enough.

First, if $H-r$ is connected, then $H$ has a spanning tree with root degree
$j$ for every $1\le j\le\deg_H(r)$. Choose $j$ neighbours of $r$, grow a
spanning forest of $H-r$ with one chosen neighbour in each component, and
join those neighbours to $r$.

Second, suppose there are $q$ different profiles consisting of one branch
that contains all vertices outside the root. These are distinct unit
vectors, so their affine dimension is $q-1$. One profile consisting only
of smaller branches raises the dimension to $q$; two distinct such profiles
raise it to $q+1$. Indeed, a coordinate describing a branch on all the vertices does
not occur in a profile with several smaller branches.

We first apply these observations at one root. The degree condition
gives enough profiles with one branch, and changing the root degree
gives profiles with several branches.

\begin{lemma}\label{lem:oneroot}
Let $d\ge3$ and let $(H,r)$ be a connected rooted graph with at least two
vertices. If every nonroot vertex has degree at least $d$, its spanning-tree
branch profiles have affine dimension at least $d$.
\end{lemma}
\begin{proof}
It is enough to prove the assertion when $C:=H-r$ is connected. Otherwise,
apply that case to one component together with $r$, and fix the choices
on all other components.

Put $t=|C|$, $D=\deg_H(r)$, and $e=e(C)$. Then
\[
 t\ge d,\qquad 2e+D\ge dt.
\]
Root $C$ at a neighbour of $r$. Lemma~\ref{lem:surplus} gives at least
$q:=e-t+2$ profiles with one branch containing all of $C$. If $D\le2$, then
$t\ge d+1$: when $t=d$, every vertex of $C$ would have to meet $r$.
The resulting bounds on $q$ are
\[
\begin{array}{c|ccc}
 D&1&2&\ge3\\ \hline
 q&\ge d+1&\ge d&\ge d-1.
\end{array}
\]
For $D=1,2$, use $q\ge((d-2)t-D)/2+2$ and the fact that $q$ is an integer.
For $D\ge3$, use $D\le t$ to obtain $q\ge(d-3)t/2+2$.

If $D=1$, these unit profiles already have dimension at least $d$.
If $D=2$, add a tree of root degree two. If $D\ge3$, add trees of root
degrees two and three. The added profiles have two or three branches,
all smaller than $C$, so the observation above gives dimension at
least $d$ in each case.
\end{proof}

We now return to the forest bound with several roots. Attaching the
same rooted tree to different roots gives different profile coordinates.
These choices usually suffice; in the remaining cases, we split $C$
into two branches.

\begin{proof}[Proof of Lemma~\ref{lem:multiroot}]
Let $a$ be the number of roots adjacent to $C$, and call a vertex of $C$
with a neighbour in $S$ a boundary vertex. Let $q$ be the number of unit
profiles $(s,R)$, where $R$ is a rooted spanning-tree type of all of $C$
at a neighbour of $s$. If $a=1$, apply Lemma~\ref{lem:oneroot}.
We may therefore assume that $a\ge2$.

As above, $q\ge d+1$ is enough. If $q=d$, we only need one more
profile, with at least two branches. Two boundary vertices give such a profile: take a spanning tree
of $C$, cut an edge on the path between them, and attach the two pieces
to roots through the two boundary vertices. The roots may coincide.

\smallskip
\noindent\emph{$C$ contains a cycle.}
Corollary~\ref{cor:cyclic} gives two rooted types at a neighbour of each
root, so $q\ge2a$. If $2\le a\le d-3$, then
$\delta(C)\ge d-a\ge3$. Lemma~\ref{lem:oneroot}, applied at a boundary
vertex for each root, gives
\[
 q\ge a(d-a+1)>d.
\]
If $a\ge d-2$, then $2a\ge d+1$, except when $(d,a)=(4,2)$.
In that case $q\ge4$, so two boundary vertices provide the one additional
profile needed. If there is only one boundary vertex $v$, then
$\deg_C(v)\ge2$, and every other vertex has degree at least four in $C$.
Thus $e(C)\ge2|C|-1$. Lemma~\ref{lem:surplus} gives at least
$|C|+1\ge3$ rooted types at $v$ for each of the two roots, so $q\ge6$.

\smallskip
\noindent\emph{$C$ is a tree.}
Each leaf of $C$ meets at least $d-1$ roots. Hence $a\ge d-1$, and there
are at least two boundary vertices. If $a\ge d+1$, then $q\ge a$ is enough.
If $a=d$, use $q\ge d$ and a profile with smaller branches.
The same argument works when $a=d-1$ and $q\ge d$.

It remains that $a=q=d-1$. Cut any edge of $C$. Each piece contains an
original leaf, and each such leaf is adjacent to all $d-1$ roots. Attach
both pieces through these leaves to one root $x$, or both to a different
root $y$. The resulting two profiles use only smaller branches and are
distinct. Together with the $d-1$ unit profiles, they give affine dimension
$d$.
\end{proof}

Only the distinction between root labels matters. Replacing them by
any other distinct labels preserves affine dimension. Later the roots
lie on one ancestor chain, so their distances from a fixed vertex will
serve as labels.

\subsection{Adding independent choices}

We have enough choices on each piece. We now count their sums, since the
pieces themselves may become indistinguishable in the whole tree. If every
piece had the same $r+1$ affinely independent choices, the sum would be
determined by how many pieces use each choice, giving
$\binom{m+r}{r}$ possibilities by stars and bars. The next lemma shows that
the same bound holds even when the available choices vary from piece to
piece.

\begin{lemma}\label{lem:affine}
Let $A_1,\ldots,A_m$ be finite subsets of $\mathbb R^U$, where $U$ is finite.
If each $A_i$ has affine dimension at least $r\ge1$, then
\[
 |A_1+\cdots+A_m|\ge\binom{m+r}{r}.
\]
\end{lemma}

The proof represents each choice by a formal exponential series.
The dimension assumption lets us choose the constant and linear terms.
Products with different lowest-degree terms then give the desired bound.

\begin{proof}[Proof of Lemma~\ref{lem:affine}]
Choose a linear map $L:\mathbb R^U\to\mathbb R^r$ such that each
$L(A_i)$ has affine dimension $r$. Such a map exists: choose $r$
linearly independent differences from each $A_i$, and require the
determinant of their images to be nonzero. Each determinant is a
nonzero polynomial in the entries of $L$, and finitely many such
polynomials can be nonzero at the same point.
In the formal power-series ring in $y_1,\ldots,y_r$, put
\[
 V_i=\operatorname{span}\{\exp\langle La,y\rangle:a\in A_i\}.
\]
The constant and linear terms of elements of $V_i$ span
$\mathbb R\oplus\mathbb R^r$. Thus $V_i$ contains series
\[
 f_{i,0}=1+O(|y|^2),\qquad
 f_{i,j}=y_j+O(|y|^2)\quad(1\le j\le r).
\]
For every $\alpha\in\mathbb Z_{\ge0}^r$ with $|\alpha|\le m$, form a
product with one factor from each $V_i$, using $f_{i,j}$ a total of
$\alpha_j$ times and $f_{i,0}$ otherwise. Its term of lowest degree is
$y^\alpha$. The resulting $\binom{m+r}{r}$ products are linearly
independent: in any nonzero relation, consider the terms of smallest
total degree.

All these products lie in the span of
\[
 \exp\langle L(a_1+\cdots+a_m),y\rangle,
 \qquad a_i\in A_i.
\]
This span has dimension at most $|A_1+\cdots+A_m|$, proving the lemma.
\end{proof}

\subsection{Switches along a DFS path}

Our third tool is global. A long DFS path already creates many spanning-tree
types, so under the assumption $\tauiso(G)=O(n^{d-1})$ every DFS tree must
have small height. The proof uses two sources of choices along a root-to-leaf
path: non-tree edges can be switched, while branches leaving the path can be
changed or reattached. The next proposition records the bound we need. Write $L(G)$ for the maximum order of a simple path in $G$. 

\begin{proposition}\label{prop:path}
There is an
absolute constant $c>0$ such that every connected $n$-vertex graph $G$ with
$\delta(G)\ge3$ satisfies
\[
 \tauiso(G)\ge n^{-2}\exp\bigl(cL(G)^{1/4}\bigr).
\]
Consequently, for fixed $d\ge3$ and $C>0$, the bound
$\tauiso(G)\le Cn^{d-1}$ implies
\[
 L(G)=O_{d,C}\bigl((\log n)^4\bigr).
\]
\end{proposition}

We prove the two sources of choices separately and then combine them.

Call a vertex of a DFS tree $T$ \emph{active} if it is incident with an
edge of $E(G)\setminus E(T)$. The next lemma does not require a
minimum-degree assumption.

\begin{lemma}\label{lem:active}
There exists $c>0$ such that the following holds. Let $T$ be a DFS spanning tree
of a connected $n$-vertex graph $G$. If a root-to-leaf path of $T$ contains $M$
active vertices, then
\[
  \tauiso(G)\ge n^{-2}\exp(cM^{1/4}).
\]
\end{lemma}

\claimsof{lem:active}
\begin{proof}
We count trees with two ordered marks. Changing the parents of distinct vertices to strict ancestors in the
original rooted tree still gives a spanning tree: each nonroot vertex
has one parent, and original depth decreases along every parent edge.
We first consider edges sharing an endpoint, then pairwise disjoint edges.

\begin{claim}\label{clm:active-star}
Suppose that $R$ non-tree edges join $x_0$ to vertices of a tree path
$x_0\cdots x_N$. For sufficiently large $R$, these edges give
$\exp(\Omega(\sqrt R))$ spanning-tree types with two ordered marks.
\end{claim}

\begin{poc}
Temporarily root $T$ at $x_0$, choose $k:=\lfloor\sqrt R/4\rfloor$ of these
edges, and replace the parent edge of each chosen endpoint by its edge to
$x_0$. Mark $x_0$ and $x_N$.

Delete the original path edges and assign to $x_i$ the positive weight equal to
the order of its resulting component. The switches split the forward path into
$k$ branches at $x_0$ not containing $x_N$, and one branch containing $x_N$.
All other branches at $x_0$ are fixed. A marked tree determines the multiset
of orders of the $k$ variable unmarked branches, by removing the branch
containing the second mark and cancelling the fixed branch orders. Any ordering
of these $k$ orders determines the cut positions uniquely, since all weights are
positive. Thus each marked type comes from at most $k!$ choices, giving
\begin{align*}
  \frac{\binom Rk}{k!}
  &\ge\left(\frac{R}{k^2}\right)^k
   \ge16^k=\exp(\Omega(\sqrt R)).
  \tag*{\qedhere}
\end{align*}
\end{poc}

Thus many edges at one vertex give the required count. If no vertex
meets many edges, we will use a large matching instead. The next claim
shows that degrees along a marked path distinguish the corresponding
switches.

\begin{claim}\label{clm:active-matching}
Suppose that $q$ pairwise disjoint non-tree edges each meet a root-to-leaf path
$P$ of $T$. For sufficiently large $q$, these edges give
$\exp(\Omega(\sqrt q))$ spanning-tree types with two ordered marks.
\end{claim}

\begin{poc}
With the original DFS root, write $P=x_0\cdots x_N$. Orient each edge as
$x_a y$, where $x_a\in P$ is a strict ancestor of $y$, and let $x_b$ be the
last vertex of $P$ on the root-to-$y$ path. Associate the interval $[a,b]$ with
this edge. Replacing the parent of $y$ by $x_a$ increases the degree of $x_a$
by one. All degree changes on $P$ lie in $[a,b]$: the only possible decrease is
at $x_{b-1}$ if $y=x_b$, or at $x_b$ if $y\notin P$. In the latter case, $P$
itself is unchanged. The case $a=b$ is allowed; the vertex $x_a$ still gains one in
degree because $x_a y$ is not a tree edge.

There are either at least $\sqrt q$ pairwise disjoint intervals or
more than $\sqrt q$ intervals containing a common point. Indeed, repeatedly choose
an interval with smallest right endpoint and delete all intervals containing
that endpoint. The chosen intervals are disjoint, and every original interval
contains one of their right endpoints. If fewer than $\sqrt q$
intervals were chosen, one of these endpoints belongs to more than
$\sqrt q$ original intervals.

In the disjoint case, switch an arbitrary subset of the corresponding edges and
mark $x_0,x_N$. For two different subsets, consider their first differing
interval along $P$. Its upper endpoint remains at the same distance from the first mark
along the marked path in both trees, since all earlier switches agree. Its
degree differs by one, since all other switches change degrees on $P$ only
within their own disjoint intervals. Thus the marked trees are non-isomorphic,
giving $2^{\Omega(\sqrt q)}$ types.

In the common-point case, let $t$ be an integer common point and discard the
matching edge incident with $x_t$, if any. All remaining upper endpoints lie strictly
before $x_t$, and every possible degree decrease on $P$ occurs at or after
$x_t$. No edge of the prefix $x_0\cdots x_t$ is deleted. Marking its endpoints,
the degrees at these distinct endpoints record the switches independently, again
giving $2^{\Omega(\sqrt q)}$ types.
\end{poc}

We now find one of the two configurations above. Choose one incident
non-tree edge for each of the $M$ active vertices of the path, and let
$F$ be the graph formed by these edges. Each edge is chosen at most twice, so
$m:=|E(F)|\ge M/2$.

If some vertex $v$ has $d_F(v)\ge\sqrt m$, then $\Omega(\sqrt M)$ of its
neighbours lie on one tree path starting at $v$. Indeed, if $v$ is outside the
original path, all its $F$-neighbours are ancestors on that path. If $v$ is on
it, at most one $F$-neighbour lies outside it, and at least half of the
remaining neighbours lie on the same side of $v$.
Claim~\ref{clm:active-star} gives $\exp(\Omega(M^{1/4}))$ marked types.
Otherwise $\Delta(F)<\sqrt m$, and a greedy matching has size at least
$\sqrt m/2=\Omega(\sqrt M)$. Claim~\ref{clm:active-matching} gives the
same bound.

Forgetting the marks loses at most a factor of $n^2$. Reducing $c$ covers
bounded positive $M$, using $\tauiso(G)\ge1$ and $n\ge3$. The case $M=0$ is
immediate.
\end{proof}

This handles active vertices. We next use the branches outside the
path. Each branch either gives two different rooted types on its own
vertices or can be attached higher up the path. Both changes can be
distinguished by marking the endpoints of the path.

\begin{lemma}\label{lem:branching}
Let $T$ be a DFS spanning tree of a connected $n$-vertex graph $G$ with
$\delta(G)\ge3$. If a root-to-leaf path $P$ of $T$ has $B$ vertices with
a child outside $P$, then
\[
 \tauiso(G)\ge n^{-2}2^B.
\]
\end{lemma}

\begin{proof}
For each such vertex $x_i$, choose a child $c_i\notin V(P)$ and let
$U_i=V(T_{c_i})$. These descendant sets are pairwise disjoint. Every
neighbour outside $U_i$ of a vertex in $U_i$ lies on the initial segment
of $P$ ending at $x_i$, by the DFS property.

If $G[U_i]$ contains a cycle, it has two distinct rooted spanning-tree
types with root $c_i$, by Corollary~\ref{cor:cyclic}. Choose either
type and retain $x_ic_i$.

If $G[U_i]$ has no cycle, it is the original tree on $U_i$.
Choose a leaf $v_i\ne c_i$ when $|U_i|\ge2$, and choose its only vertex
otherwise. The vertex $v_i$ has at least two neighbours outside $U_i$,
so one, say $a_i$, lies strictly above $x_i$ on $P$. Either retain
$x_ic_i$, or replace it by $a_iv_i$, keeping the tree on $U_i$ fixed.
This case cannot occur at the DFS root, since there would then be
only one possible outside neighbour.

Every choice gives a spanning tree: each $U_i$ induces a tree attached
exactly once to the unchanged tree on its complement. Mark the two
endpoints of $P$ in order. Moving a branch whose induced graph has no cycle changes the degree
vector along the marked path by $e_{a_i}-e_{x_i}$. These vectors are
linearly independent. Indeed, the deepest vertex $x_i$ with nonzero
coefficient in a linear relation receives no contribution from a
vertex $x_j$ above it, and the coefficients of all vertices $x_j$ below it vanish.
Its coordinate is therefore the negative of its own nonzero
coefficient. Changing only the tree inside a branch does not change degrees on
$P$. Thus the degree vector determines exactly which branches were moved.

After these choices are recovered, cancel their branches, and all
other fixed branches, from the multiset of rooted branches at each
vertex of the marked path. For each $U_i$ such that $G[U_i]$ contains a cycle, the remaining
branch at $x_i$ determines its chosen rooted type. Thus the $2^B$ choices give
distinct types with two ordered marks. Forgetting the marks loses at
most a factor $n^2$.
\end{proof}

The two bounds now cover every vertex of a DFS path. If a vertex has
no off-path child, then its tree degree is at most two, so minimum
degree three forces a non-tree edge. The two counts can therefore be combined directly.

\begin{proof}[Proof of Proposition~\ref{prop:path}]
On a root-to-leaf DFS path of order $H$, let $B$ count the vertices
with an off-path child, and let $M$ count the active vertices. Every
vertex not counted by $B$ has tree degree at most two and hence is
active, so $H\le B+M$. Lemmas~\ref{lem:branching} and~\ref{lem:active}
give respectively $n^{-2}2^B$ and
$n^{-2}\exp(c_0M^{1/4})$ as lower bounds. Since
$\max(B,M)\ge H/2$, these imply the stated bound with $H$ in place
of $L(G)$. Any given simple path can be followed as the initial
descent of a DFS by always exploring its next vertex first. It lies
on a root-to-leaf DFS path, so we may apply the bound to a longest
simple path. Taking logarithms gives the stated bound on $L(G)$.
\end{proof}

\section{Stability}\label{sec:stability}

In this section we prove Theorem~\ref{thm:stability}. The proof rests on two
propositions, which we state first. The first controls a single
common-neighbourhood class. It is used in both directions: a large class has
only a bounded remainder, while a class below $n/2$ must in fact be
sublinear.

\begin{proposition}\label{prop:amplify}
For $d\ge3$ there is $b_d>0$ such that the following holds.
Let $G$ be connected of order $n$, with $\delta(G)\ge d$, and let
 $X=X(S)$, $|S|=d$, $m=|X|\ge1$ and $r=n-d-m$.
Then
\begin{equation}\label{eq:amplification}
 \tauiso(G)\ge b_d\frac{m^d(r+1)}{n}.
\end{equation}
In particular, fix $C>0$ and assume that $\tauiso(G)\le Cn^{d-1}$.
For sufficiently large $n$, a class with $m\ge n/2$ has
$r=O_{d,C}(1)$, while a class with $m<n/2$ satisfies
 $$m\le K_{d,C}n^{1-1/d}$$
for a constant $K_{d,C}>0$.
\end{proposition}

The second proposition says that a large class must exist whenever the
number of spanning-tree types is only of order $n^{d-1}$.

\begin{proposition}\label{prop:large-class}
Fix $d\ge3$ and $C>0$. For all sufficiently large $n$, every connected
$n$-vertex graph $G$ with $\delta(G)\ge d$ and
$\tauiso(G)\le Cn^{d-1}$ has a set $S$ of $d$ vertices with
$$|X(S)|\ge n/2.$$
\end{proposition}

Theorem~\ref{thm:stability} is now immediate from these two propositions.

\begin{proof}[Proof of Theorem~\ref{thm:stability}]
Fix $d\ge3$ and $C>0$, and let $n$ be sufficiently large.
Proposition~\ref{prop:large-class} gives a set $S$ of $d$ vertices with
$m:=|X(S)|\ge n/2$. Put $r=n-d-m$. Proposition~\ref{prop:amplify} yields
\[
  Cn^{d-1}\ge\tauiso(G)\ge b_d\frac{m^d(r+1)}{n}
  \ge \frac{b_d}{2^d}n^{d-1}(r+1).
\]
Thus $r=O_{d,C}(1)$, which proves the theorem.
\end{proof}

It remains to prove Propositions~\ref{prop:amplify} and~\ref{prop:large-class}.
We do this in the next two subsections.

\subsection{Proof of Proposition~\ref{prop:amplify}}

The proof combines two independent choices: a rooted tree on the remainder
and the numbers of leaves attached to the common neighbours. We first give
a slightly more general lemma, since the same estimate will later be used
for groups of DFS leaves.

\begin{lemma}\label{lem:pool}
For every fixed $k\ge3$ there is $a_k>0$ with the following property.
Let $G$ have order $n$, and let $X,S$ be disjoint sets with $|X|=m\ge1$
and $|S|=k$. Suppose $X$ is independent, every vertex of $X$ is adjacent
to every vertex of $S$, and every component of $G-X$ meets $S$.
Let $(J,\rho)$ be obtained from $G-X$ by identifying $S$ to a root and
deleting loops and repeated edges. Then
\[
 \tauiso(G)\ge a_k\frac{m^k}{n}\,q(J,\rho).
\]
\end{lemma}
\claimsof{lem:pool}
\begin{proof}
We separate the construction of marked trees from the loss caused by
forgetting the marks.

\begin{claim}
For all sufficiently large $m$, the construction below gives
$\Omega_k(q(J,\rho)m^{k-1})$ distinct spanning-tree types with a marked
center and an ordering of its $k$ neighbours.
\end{claim}
\begin{poc}
For each rooted spanning-tree type of $J$, choose one representative.
Replace each edge at its root by a corresponding edge of $G-X$ to a
vertex of $S$. This gives an $S$-rooted spanning forest on $G-X$; keep
one such forest for each chosen type. Reserve $b_0\in X$, join it to every
vertex of $S$, and attach the other vertices of $X$ as leaves at $S$.
We call the vertices of $S$ hubs. For large $m$, there are
$\Omega_k(m^{k-1})$ leaf vectors giving every hub at least $m/(2k)$ leaves.

Mark $b_0$ and order its neighbours. After deleting $b_0$, every branch
at a hub with at least two vertices belongs to the chosen forest. Put
$r=n-k-m$. Subtracting the total order of these branches from $r$ gives
the number of one-vertex remainder branches. Identifying the hubs now
recovers the rooted $J$-type. The forest chosen for that type was fixed,
so the ordered hub degrees determine the leaf vector. Thus all these
marked trees are distinct.
\end{poc}

We have counted the marked trees. It remains to control how much information
can be lost when the center and the ordering of its neighbours are forgotten.

\begin{claim}
For all sufficiently large $m$, each unrooted tree produced above comes
from at most $O_k(n/m)$ choices of the marked center and the ordering of
its neighbours.
\end{claim}
\begin{poc}
In any resulting tree, let $Q$ be the set of vertices of degree at least
$m/(2k)$. The degree sum gives $|Q|\le4kn/m$. Every possible image of
$b_0$ has degree $k$, lies outside $Q$, and has all its neighbours in
$Q$. If $B$ is the set of these possible images, then the edges between
$B$ and $Q$ form a forest. Hence
\[
 k|B|\le |B|+|Q|-1,\qquad
 |B|\le\frac{|Q|-1}{k-1}=O_k(n/m).
\]
Each center has at most $k!$ neighbour orderings, proving the claim.
\end{poc}

Combining the two claims gives
\[
 \tauiso(G)\ge a_k\frac{m^k}{n}\,q(J,\rho)
\]
for all sufficiently large $m$, after decreasing $a_k$ if necessary.
For bounded $m$, use one fixed leaf vector for every chosen forest.
The same recovery gives at least $q(J,\rho)$ marked types, and there
are at most $nk!$ choices of marks. A further decrease of $a_k$ covers
these cases as well.
\end{proof}

The lemma already bounds a group of DFS leaves with common selected
neighbours: deleting those leaves leaves a connected tree containing
$S$, and we may use $q(J,\rho)\ge1$. For $X=X(S)$, the set $X$ contains every vertex with neighbourhood
$S$. Together with the degree condition, this gives enough edges in
the contracted remainder for Lemma~\ref{lem:surplus} to provide the
extra factor $r+1$.

\begin{proof}[Proof of Proposition~\ref{prop:amplify}]
Write $R=V(G)\setminus(S\cup X)$. There are no edges between $X$ and $R$.
No vertex of $R$ is isolated in $G[R]$: otherwise all its neighbours would
lie in $S$, and its degree would put it in $X$.

Contract $S$ in $G-X$ to a root $\rho$, and delete loops and repeated
edges. The resulting graph $J$ is connected and has $r+1$ vertices.
A vertex of $R$ with no neighbour in $S$ retains degree at least $d$.
A vertex with a neighbour in $S$ also has a neighbour in $R$, so its degree
in $J$ is at least two. Thus every nonroot vertex has degree at least two,
and every degree-two nonroot vertex is adjacent to $\rho$.

If $t$ nonroot vertices have degree two, then $\deg_J(\rho)\ge t$.
Consequently,
\[
 2e(J)\ge 2t+3(r-t)+t=3r.
\]
Lemma~\ref{lem:surplus} gives
\[
 q(J,\rho)\ge e(J)-r+1\ge r/2+1.
\]
Apply Lemma~\ref{lem:pool} with $k=d$ and $b_d=a_d/2$ to obtain
\eqref{eq:amplification}.

Now assume $\tauiso(G)\le Cn^{d-1}$. If $m\ge n/2$, then
\eqref{eq:amplification} gives $r+1\le 2^dC/b_d$.
If $m<n/2$ and $n$ is sufficiently large, then $r+1\ge n/3$, so
$m^d\le(3C/b_d)n^{d-1}$. This proves both conclusions.
\end{proof}

\subsection{Proof of Proposition~\ref{prop:large-class}}

Assume throughout this subsection that every degree-$d$
common-neighbourhood class has fewer than $n/2$ vertices. By
Proposition~\ref{prop:amplify}, every such class is then sublinear.
We will derive a contradiction in two steps: first we find many equal-size
nontrivial pieces in a DFS tree, and then we show that changing these pieces
already gives too many tree types.

At each nonleaf vertex of a rooted tree, choose a child whose descendant
subtree has maximum order and call it the \emph{heavy child}; call all other
children \emph{light}. We fix ties once and for all. A \emph{light subtree}
is the descendant subtree rooted at a light child. Every light child has a
heavy sibling whose descendant subtree is at least as large.

\subsubsection{Step 1: find many equal-size pieces.}

\begin{lemma}\label{lem:many-light}
Fix $d\ge3$ and $C>0$. For all sufficiently large $n$, let $G$ be a
connected $n$-vertex graph with $\delta(G)\ge d$ and
$\tauiso(G)\le Cn^{d-1}$. Suppose that every degree-$d$
common-neighbourhood class has fewer than $n/2$ vertices.
Then every DFS tree of height $H$ has at least
\[
 N\ge c_{d,C}\frac{n}{H^{2d+5}}
\]
light children whose descendant subtrees all have one common order
$b\ge2$, where $c_{d,C}>0$ depends only on $d,C$.
\end{lemma}

The proof has two steps. We first show that the DFS tree has many
nonleaf vertices. We then use only the tree structure to select many light
subtrees of one common order.

\begin{lemma}\label{lem:nonleaf}
Under the hypotheses of Lemma~\ref{lem:many-light}, every DFS tree
of height $H$ has at least
\[
 I\ge c_{d,C}\frac{n}{H^{d+1}}
\]
nonleaf vertices, for some constant $c_{d,C}>0$.
\end{lemma}
\claimsof{lem:nonleaf}
\begin{proof}
If $I\ge n/2$, there is nothing to prove. Otherwise the tree has at least
$n/2$ leaves. We use the path bound
$H=O_{d,C}((\log n)^4)$ from Proposition~\ref{prop:path}.

\begin{claim}
There is a set of at least $n/(4H^{d+1})$ leaves which can be grouped by
their parents so that every group has size at most
$K_{d,C}n^\beta$, where $\beta=d/(d+1)$.
\end{claim}
\begin{poc}
For a leaf of graph degree $d$, select all its neighbours. For any other
leaf, select $d+1$ neighbours, including its tree parent. Let
$k\in\{d,d+1\}$ be the number selected, and record their increasing tuple
of ancestor depths. There are at most $2H^{d+1}$ tuples, so one occurs for
at least $n/(4H^{d+1})$ leaves. Group these leaves by their parent $w$.
Within each group, the selected neighbours are the same: their vertices
are determined by $w$ and the recorded depths.

If $k=d$, every group lies in a degree-$d$ common-neighbourhood class, so
Proposition~\ref{prop:amplify} bounds its size by
$O_{d,C}(n^{1-1/d})$. If $k=d+1$, Lemma~\ref{lem:pool}, with
$q(J,\rho)\ge1$, bounds its size by $O_{d,C}(n^{d/(d+1)})$.
Its hypotheses hold because the group consists of DFS leaves. Since
$d/(d+1)>1-1/d$, the stated bound follows.
\end{poc}

Thus we have many leaves arranged in groups whose individual sizes are
controlled. We next show that only few of these groups can still be large.

\begin{claim}
Put $R=2d(d+1)$. The number $t$ of groups with at least $R$ leaves is
$O_{d,C}(n^{d/R})$. Consequently, all such large groups together contain
$o(n/H^{d+1})$ leaves.
\end{claim}
\begin{poc}
Select $R$ leaves at each of the $t$ corresponding parents $w$. For
$0\le y_w\le R$, keep $y_w$ of these leaves at $w$ and move the others
to their shallowest selected neighbour.

All the parents have one common depth $j$. Since $k\ge3$, each moved leaf
ends strictly above level $j$. The degree multiset at level $j$ is
therefore a fixed multiset plus $\{b_w+y_w\}$, with fixed integers $b_w$.
Lemma~\ref{lem:affine}, applied to the $R$-dimensional sets
$\{e_{b_w},\ldots,e_{b_w+R}\}$, gives at least
$\binom{t+R}{R}$ rooted tree types. There are at most
$n\tauiso(G)\le Cn^d$ rooted types altogether, so
$t=O_{d,C}(n^{d/R})$. By the first claim, the number of leaves in the
large groups is at most
\begin{align*}
 O_{d,C}(n^{\beta+d/R})
 &=O_{d,C}(n^{1-1/(2(d+1))})
  =o(n/H^{d+1}).
 \tag*{\qedhere}
\end{align*}
\end{poc}

After the second claim, almost all leaves in the chosen pattern lie in groups
of bounded size. We can therefore finish by counting their distinct parents.

The fixed tuple occurs for at least $n/(4H^{d+1})$ leaves. After removing
the large groups, at least $n/(8H^{d+1})$ remain for large $n$. Each
remaining group has fewer than $R$ leaves and a different parent. Hence
there are at least $c_{d,C}n/H^{d+1}$ nonleaf vertices, as required.
\end{proof}

We now have many nonleaf vertices. Partitioning the tree into heavy
paths, each with at most $H$ vertices, gives many paths starting at
nonleaf light children. The next lemma discards very large subtrees
and chooses a common order among those left. No degree condition is
needed for this step.

\begin{lemma}\label{lem:light-extraction}
Let $T$ be a rooted tree of order $n$, height $H$, and with $I$ nonleaf
vertices. For every integer $B\ge2$, some $b\in\{2,\ldots,B\}$ occurs as
the order of the descendant subtree of at least
\[
 \frac{I/H-1-nH/B}{B}
\]
light children, provided this expression is positive.
\end{lemma}
\begin{proof}
Partition the tree into paths that follow heavy-child edges. Every
such path contains at most $H$ nonleaf vertices. Except for the
path starting at the root, every heavy path containing a nonleaf vertex
starts at a nonleaf light child. There are therefore at least $I/H-1$
light children whose descendant subtrees have at least two vertices.
Also,
\[
 \sum_{v\in V(T)}|T_v|
 =\sum_{x\in V(T)}|\{\text{ancestors of }x\text{, including }x\}|
 \le nH.
\]
At most $nH/B$ descendant subtrees have order greater than $B$.
Grouping the remaining light subtrees by their orders gives the stated
bound.
\end{proof}

We choose $B$ large enough that discarding subtrees with more than
$B$ vertices removes only a small part of the count. Since $H$ is
bounded by a power of $\log n$, choosing one common order still leaves
enough subtrees.

\begin{proof}[Proof of Lemma~\ref{lem:many-light}]
By Lemma~\ref{lem:nonleaf}, the DFS tree has
$I\ge cn/H^{d+1}$ nonleaf vertices, with $c=c_{d,C}>0$. Set
 $B=\left\lceil\frac{4H^{d+3}}{c}\right\rceil.$
Then $nH/B\le I/(4H)$. Since $H=O_{d,C}((\log n)^4)$ by
Proposition~\ref{prop:path}, we also have $I/H\to\infty$.
For sufficiently large $n$, Lemma~\ref{lem:light-extraction} gives
\begin{align*}
  N\ge\frac{I}{2HB}\ge c'_{d,C}\frac{n}{H^{2d+5}}.
  \tag*{\qedhere}
\end{align*}
\end{proof}

\subsubsection{Step 2: count the pieces.}

The selected subtrees are disjoint and can be changed independently.
The next lemma shows that these changes already give too many unrooted
tree types.

\begin{lemma}\label{lem:guarded}
Suppose $\delta(G)\ge d\ge3$, and a DFS tree of height $H$ has $N$
light children whose descendant subtrees all have one order $b\ge2$.
There is a constant $c_d>0$, depending only on $d$, such that if
$N/H^2$ is sufficiently large in terms of $d$, then at least one of
the following holds:
\[
 \tauiso(G)\ge c_d(N/H)^d,
 \qquad
 \tauiso(G)\ge c_d\frac{(N/H^2)^{d+1}}{n}.
\]
\end{lemma}

There are two issues. We first recover the sum of the local profiles from
the whole rooted tree. We then choose the root so that forgetting it costs
little. Lemma~\ref{lem:recovery} handles the first point; the rest of the
proof uses either a centroid or a high-degree root.

\begin{lemma}\label{lem:recovery}
Let $G$ be a graph with a partition
\[
 V(G)=B\mathbin{\dot\cup}U_1\mathbin{\dot\cup}\cdots
       \mathbin{\dot\cup}U_m,
\]
where each $G[U_i]$ is connected and $|U_i|=b\ge2$.
Let $T_0$ be a tree on $B$, rooted at $q$, with edges in $G$.
Suppose every vertex of each $U_i$ has degree at least $d\ge3$ in $G$, and
\[
 S_i:=N_G(U_i)\setminus U_i\subseteq B.
\]
For each $i$, assume that $S_i$ lies on the $q$--$p_i$ path in $T_0$ for
some $p_i\in S_i$, and that the descendant subtree of $T_0$ at $p_i$ has
more than $b$ vertices.
Keep $T_0$ and, independently for each $i$, add any $S_i$-rooted spanning
forest on $U_i\cup S_i$. The resulting trees have at least
$\binom{m+d}{d}$ types rooted at $q$.
\end{lemma}
\claimsof{lem:recovery}
\begin{proof}
Each component of a chosen forest meets $B$ at exactly one root, so
adding the forests to $T_0$ always gives a spanning tree. By
Lemma~\ref{lem:multiroot}, the set of branch profiles for each $U_i$ has
affine dimension at least $d$. It remains to show that their sum can be
recovered from the rooted tree.

Let $F$ be a resulting tree, rooted at $q$. Delete every vertex $v$ for
which $|F_v|>b$. Root each remaining component at its vertex closest to
$q$, and record its rooted type and the depth of its root in $F$.
This record depends only on the rooted isomorphism type of $F$.

\begin{claim}
Among the vertices of the fixed tree $T_0$, the set of deleted vertices
and all remaining rooted components are independent of the chosen local
forests.
\end{claim}
\begin{poc}
Let
\[
  P_*:=\bigcup_{i=1}^m V(T_0[q,p_i]),
\]
where $T_0[q,p_i]$ is the path from $q$ to $p_i$ in $T_0$.
Every vertex of $P_*$ is an ancestor of some $p_i$. Its descendant
subtree in $T_0$ therefore has more than $b$ vertices, so it is deleted
for every choice of forests. If $v\in B\setminus P_*$, no attachment
vertex lies in its descendant subtree in $T_0$: otherwise $v$ would lie
on a path from $q$ to some $p_i$. No chosen branch is therefore added
below $v$, and its descendant subtree stays unchanged. This proves the
claim.
\end{poc}

The fixed part of the record is now under control. We next check that the
variable branches survive the deletion exactly in the form recorded by their
local profiles.

\begin{claim}
Each branch of a chosen forest appears as one recorded component, and the
whole record is a fixed vector plus the sum of the local branch profiles,
with attachment vertices relabelled by their depths from $q$.
\end{claim}
\begin{poc}
A branch of a chosen forest has at most $b$ vertices, so none of its
vertices is deleted. Its attachment vertex belongs to some
$S_i\subseteq P_*$ and is deleted. Hence the whole branch is exactly one
of the recorded components.

Thus each label $(s,R)$ in a local profile is replaced by
$(\operatorname{dist}_{T_0}(q,s)+1,R)$. Distinct vertices of $S_i$ have
distinct distances from $q$, since they lie on one path starting at $q$.
All other recorded components come from the fixed tree and are independent
of the local choices. This proves the claim.
\end{poc}

The relabelling in the second claim preserves the affine dimension of
each profile set. Lemma~\ref{lem:affine} therefore gives at least
$\binom{m+d}{d}$ different records, and hence that many rooted types.
Only the sum of the local profiles is recovered; the individual pieces
need not be identified.
\end{proof}

We now prepare to apply the lemma. Keeping light subtrees at one
depth leaves their heavy siblings unchanged. These siblings provide
the fixed vertices needed below the attachment paths. We then choose
a root that either stays a centroid or has large degree in every
resulting tree.

\claimsof{lem:guarded}
\begin{proof}[Proof of Lemma~\ref{lem:guarded}]
Let $T$ be the given DFS tree, with original root $\rho$.

\begin{claim}
There are $M\ge N/H$ of the selected light subtrees whose roots have one
common depth. Writing $U_i=V(T_{u_i})$, with parent $p_i$, each $p_i$ has
a heavy-child subtree $W_i$ of order at least $b$ that stays fixed when
the $U_i$ are changed. Moreover, if
$S_i=N_G(U_i)\setminus U_i$ and $a_i$ is its shallowest vertex, then
$S_i$ lies on the ancestor chain from $a_i$ to $p_i$.
\end{claim}
\begin{poc}
Keep $M\ge N/H$ light subtrees whose roots have the same depth. Let $W_i$
be the vertex set of the subtree at the heavy child of $p_i$. Then
$|W_i|\ge b$. No $W_i$ meets a kept $U_j$: their roots are at the same
depth, but the root of $W_i$ is a heavy child and $u_j$ is a light child.
Thus every $W_i$ stays fixed.

The DFS property puts $S_i$ on the strict ancestor chain of $u_i$, with
$p_i\in S_i$. In particular, no vertex of $S_i$ lies in a kept subtree,
and $S_i$ lies between $a_i$ and $p_i$ on that ancestor chain.
\end{poc}

Choose a centroid $z$ of $T$, and let $P$ be the $\rho$--$z$ path. Discard
the at most one kept subtree containing $z$. Every other kept subtree
avoids $P$.

\begin{claim}
If at least half of the remaining candidates satisfy $a_i\notin V(P)$,
then
\[
  \tauiso(G)\ge c_d(N/H)^d.
\]
\end{claim}
\begin{poc}
Vary these $m\ge(M-1)/2$ subtrees, and let $T_0$ be the tree obtained
from $T$ by deleting their vertex sets, rooted at $z$. For each such $i$,
the path from $z$ to $p_i$ follows $P$ up to its meeting point with the
ancestor chain of $p_i$, and then follows that chain down to $p_i$.
Since $a_i$ lies outside $P$, the whole set $S_i$, lying between $a_i$
and $p_i$, lies on the $z$--$p_i$ path.

Also, $p_i\notin V(P)$, so $z$ is outside its original descendant
subtree. The vertices of $W_i$ remain below $p_i$ after rooting at $z$,
and none has been deleted from $T_0$. Hence the descendant subtree of
$T_0$ at $p_i$ has at least $|W_i|+1>b$ vertices. Lemma~\ref{lem:recovery}
gives $\binom{m+d}{d}$ rooted types.

Each variable subtree and all its attachment vertices lie in one
component of $T-z$. Every replacement keeps the vertex sets of these
components unchanged, so $z$ stays a centroid. Since a tree has at most
two centroids, forgetting the root loses at most a factor of two. Thus
\begin{align*}
  \tauiso(G)\ge\frac12\binom{m+d}{d}\ge c_d(N/H)^d.
  \tag*{\qedhere}
\end{align*}
\end{poc}

If this first case does not occur, many shallowest attachment vertices must
lie on the centroid path. We use this concentration to force a high-degree
root and obtain the second bound.

\begin{claim}
If fewer than half of the remaining candidates satisfy $a_i\notin V(P)$,
then
\[
  \tauiso(G)\ge c_d\frac{(N/H^2)^{d+1}}{n}.
\]
\end{claim}
\begin{poc}
At least half of the candidates now have $a_i\in V(P)$. Since $|P|\le H$,
some $a\in V(P)$ equals $a_i$ for at least
\[
  K\ge\frac{M-1}{2H}
\]
candidates. Split them into a reserved set and a variable set of sizes
$r,m\ge K/3$, which is possible for large $K$.

Construct the fixed tree $T_0$ as follows. Start with $T$ and delete the
variable subtrees. For each reserved subtree, remove its original
attachment edge, choose a fixed internal spanning tree, and attach it once
at $a$ using an edge of $G$. Such an edge exists because its shallowest
attachment vertex is $a$. These changes preserve a tree; keep this tree
fixed and root it at $a$.

The reserved subtrees avoid every relevant ancestor chain and every $W_i$.
Reattaching them therefore changes neither the path from $a$ to a variable
$p_i$ nor the subtree on $W_i$ and its edge to $p_i$. Since $a=a_i$, all
of $S_i$ lies on the unchanged $a$--$p_i$ path, while $W_i$ still lies
below $p_i$. Hence Lemma~\ref{lem:recovery} again gives
$\binom{m+d}{d}$ rooted types.

The reserved subtrees ensure that $\deg(a)\ge r$ in every resulting tree.
An $n$-vertex tree has at most $2(n-1)/r$ vertices of degree at least $r$.
Forgetting the root therefore gives
\begin{align*}
  \tauiso(G)
  &\ge\frac{r}{2(n-1)}\binom{m+d}{d}
   \ge c_d\frac{K^{d+1}}{n}
   \ge c'_d\frac{(N/H^2)^{d+1}}{n}.
  \tag*{\qedhere}
\end{align*}
\end{poc}

The two claims cover all possibilities, and so one of the two bounds in
the statement of the lemma always holds.
\end{proof}

\subsubsection{Finishing Proposition~\ref{prop:large-class}}
We can now finish the proof of Proposition~\ref{prop:large-class}.
Assuming that every class is small, Lemma~\ref{lem:many-light} gives
almost linearly many equal-size pieces, up to powers of the DFS height,
and Lemma~\ref{lem:guarded} turns them into $n^d$ types up to logarithmic
factors.

\begin{proof}[Proof of Proposition~\ref{prop:large-class}]
Fix $d,C$ and suppose $\tauiso(G)\le Cn^{d-1}$.
By Proposition~\ref{prop:path}, every DFS tree has height
 $H=O_{d,C}((\log n)^4).$
Suppose no degree-$d$ common-neighbourhood class has at least $n/2$
vertices. Lemma~\ref{lem:many-light} supplies $N$ light subtrees
of one common order $b\ge2$, where
 $N\ge c_{d,C}\frac{n}{H^{2d+5}}.$
In particular, $N/H^2\to\infty$. Either alternative of
Lemma~\ref{lem:guarded} gives
\[
 \tauiso(G)\ge\frac{c'_{d,C}n^d}{(\log n)^{O_d(1)}},
\]
contrary to $\tauiso(G)\le Cn^{d-1}$ for sufficiently large $n$.
Thus some $X(S)$ with $|S|=d$ has at least $n/2$ vertices, as required.
\end{proof}

\section{The extremal comparison}\label{sec:extremality}

By stability, every possible minimizer has a degree-$d$
common-neighbourhood class with only a bounded remainder. This section
compares all such graphs. The first proposition gives the bipartite
baseline; the second is the strict comparison needed for uniqueness.

\begin{proposition}\label{prop:coefficient}
For every fixed integer $d\ge2$, as $n\to\infty$,
\[
 \tauiso(K_{d,n-d})=A_dn^{d-1}+O_d(n^{d-2}).
\]
\end{proposition}

To prove uniqueness, we need more than the count for the example: the other
graphs must give an increase larger than the error term. The next
proposition gives exactly this comparison, uniformly when the remainder is
bounded.

\begin{proposition}\label{prop:gap}
For every fixed integer $d\ge3$ and $r_0\ge0$, there is
$n_0=n_0(d,r_0)$ such that the following holds. Let $G$ be a connected
$n$-vertex graph with $n\ge n_0$ and $\delta(G)\ge d$. Suppose that
$|S|=d$, $X=X(S)$, and $|V(G)\setminus(S\cup X)|\le r_0$.
If $G\not\cong K_{d,n-d}$, then
\[
  \tauiso(G)>\tauiso(K_{d,n-d}).
\]
\end{proposition}

Both proofs describe a spanning tree by a small core joining the
$d$ common neighbours and the branches attached to them. We first
count all bipartite cores, then use a remainder or an extra edge to
produce additional types that do not overlap with the first family.

\subsection{The complete bipartite graph}

Johnson and Nochumson~\cite{johnson} give general bounds for complete
bipartite graphs. Here the smaller part has fixed order, so deleting
the leaves on the larger side leaves only finitely many possible
cores. We count labelled cores and leaf vectors first, then account
for the different labellings that give the same tree type.

\claimsof{prop:coefficient}
\begin{proof}[Proof of Proposition~\ref{prop:coefficient}]
Put $m=n-d>d$, and let $S$ be the part of order $d$.

\begin{claim}
For every spanning tree, deleting all leaves in the part of order $m$
while keeping all vertices of $S$ leaves a core $F$ with
$1\le k:=|F-S|\le d-1$.
\end{claim}
\begin{poc}
The bipartition of a spanning tree is determined by the tree, and $S$ is
its smaller class. Every vertex of $F-S$ has degree at least two. Since
\[
 \sum_{v\notin S}(\deg_T(v)-1)=d-1,
\]
the number of vertices of $F-S$ is at most $d-1$, and it is at least one
because the tree is connected.
\end{poc}

Thus only the finitely many values $1\le k\le d-1$ can occur. For a fixed
$k$, we now count the labelled cores and the possible leaf multiplicities.

\begin{claim}
For fixed $k$, the number of labelled cores is
\[
 d^{k-1}k!S(d-1,k),
\]
and the number of possible leaf vectors is
\[
 \binom{m-k+d-1}{d-1}
 =\frac{m^{d-1}}{(d-1)!}+O_d(m^{d-2}).
\]
\end{claim}
\begin{poc}
Label the $d$ vertices of $S$ and the $k$ vertices of $F-S$. The
bipartite Pr\"ufer encoding gives
\begin{equation}\label{eq:labeled-cores}
 d^{k-1}k!S(d-1,k)
\end{equation}
cores: the word on the $k$-part has length $d-1$ and uses every letter,
while the word on the $d$-part has length $k-1$.

The remaining leaves are described by a nonnegative vector
$x=(x_1,\ldots,x_d)$ with $\sum_i x_i=m-k$, which gives the stated
stars-and-bars count.
\end{poc}

This counts labelled objects. To pass to unlabelled tree types, it remains to
understand when different labellings can represent the same tree.

\begin{claim}
All but $O_d(m^{d-2})$ leaf vectors give tree types with exactly $d!k!$
labelled representations $(F,x)$.
\end{claim}
\begin{poc}
Only $O_d(m^{d-2})$ leaf vectors have two equal coordinates. For every
other vector, a permutation preserving $(F,x)$ must fix every vertex of
$S$. It must also fix every vertex of $F-S$: two such vertices cannot
have the same neighbourhood in $S$, since their degrees are at least two
and that would give a four-cycle in $F$. Thus the stabilizer is trivial,
and the orbit has size $d!k!$.
\end{poc}

By the last two claims, the contribution for this value of $k$ is
\[
 \frac{d^{k-1}S(d-1,k)}{d!(d-1)!}m^{d-1}+O_d(m^{d-2}).
\]
Summing over $1\le k\le d-1$ and using $m=n-d$ proves the proposition.
\end{proof}

Every core contributes to $A_d$. The star joining all $d$ hubs gives
only $1/(d!(d-1)!)$. We keep all cores in the next proof so that the
first family already has coefficient $A_d$; the remainder or an extra
edge will then provide additional types.

\subsection{A strict gap for every other graph}

There are two cases. A nonempty remainder gives a cycle after $S$
is contracted, and hence two different rooted tree types. Combining
each with all bipartite cores gives two disjoint families with the
same leading coefficient. If the remainder is empty, an extra edge
inside $S$ gives a core that does not occur in the bipartite family.

\claimsof{prop:gap}
\begin{proof}[Proof of Proposition~\ref{prop:gap}]
Write $m=|X|$, $R=V(G)\setminus(S\cup X)$, and $r=|R|$.
Set $L=d+\lfloor r_0\rfloor+1$, so $L>d+r_0$ and depends only on
$d,r_0$.

\begin{claim}
There are $q$ fixed $S$-rooted forests on $R\cup S$, where $q=2$ if
$r>0$ and $q=1$ if $r=0$, such that in the first case their contractions
have different rooted spanning-tree types.
\end{claim}
\begin{poc}
Every component of $G-X$ meets $S$, since there are no edges from $X$ to
$R$. Hence an $S$-rooted spanning forest on $R\cup S$ exists. If $r>0$,
contract $S$ to one root in $G-X$, and delete loops and repeated edges.
The resulting graph $J$ is the one used in the proof of
Proposition~\ref{prop:amplify}. Every nonroot vertex has degree at least
two, so $J$ contains a cycle. Corollary~\ref{cor:cyclic} gives two rooted
spanning-tree types. Choose a representative of each, and replace its
edges at the root by corresponding edges to $S$. If $r=0$, use the empty
remainder forest.
\end{poc}

We now attach the large common-neighbourhood class to each of these fixed
remainder forests. The next claim shows that each such choice carries the full
bipartite leading term.

\begin{claim}
For each chosen remainder forest, the bipartite-core construction gives
\[
 A_dn^{d-1}-O_{d,r_0}(n^{d-2})
\]
distinct tree types. If $r>0$, the two families coming from the two
remainder forests are disjoint. Consequently,
\[
 \tauiso(G)\ge qA_dn^{d-1}-O_{d,r_0}(n^{d-2}).
\]
\end{claim}
\begin{poc}
For each $1\le k\le d-1$, take a bipartite core $F$ with its $d$ hubs
and $k$ other vertices labelled. Use $k$ vertices of $X$ for these
non-hub vertices, add the fixed remainder forest, and attach the other
vertices of $X$ as leaves, with $x_i\ge L$ leaves at each hub. By
\eqref{eq:labeled-cores}, there are $d^{k-1}k!S(d-1,k)$ labelled cores.
The number of leaf vectors is
\[
  \binom{m-k-dL+d-1}{d-1}
  =\frac{n^{d-1}}{(d-1)!}+O_{d,r_0}(n^{d-2}).
\]

The hubs are exactly the vertices of degree greater than $d+r_0$: each
has at least $L$ leaves, whereas all other vertices have degree at most
$d+r_0$. The smallest subtree containing the hubs is $F$, since each
remainder component is attached at only one hub and every non-hub vertex
of $F$ has degree at least two. Thus for a fixed forest and $k$, each
unrooted type comes from at most $d!k!$ labelled choices. Summing over
$k$ gives the stated contribution $A_dn^{d-1}-O_{d,r_0}(n^{d-2})$.

If $r>0$, once the core is identified, every branch outside it with at
least two vertices belongs to the remainder. Their total order,
subtracted from $r$, gives the number of one-vertex remainder branches.
After identifying the hubs as one root, these branches recover the rooted
$J$-type. Hence the two remainder forests give disjoint families.
\end{poc}

This already gives a strict gain when $r>0$. If $r=0$, there is only the
empty remainder forest, so the extra contribution must come from an edge
inside $S$.

\begin{claim}
If $r=0$ and $G\not\cong K_{d,n-d}$, then there is an additional family
of tree types, disjoint from the bipartite-core family, of size at least
\[
  \frac{n^{d-1}}{d!(d-1)!}-O_d(n^{d-2}).
\]
\end{claim}
\begin{poc}
Since $r=0$, the graph $G$ is $K_{d,n-d}$ with at least one extra edge
$ab$ inside $S$. Choose $z\in X$ and use the tree on $S\cup\{z\}$ with
edges
\[
  ab\quad\text{and}\quad zs\ \text{ for }s\in S\setminus\{b\}.
\]
Attach the remaining vertices of $X$ as leaves, again with $x_i\ge L$.
The hubs are identified by their degrees. Their smallest connecting
subtree now contains an edge between two hubs, so none of these types
belongs to the bipartite-core family. Each type comes from at most $d!$
labelled leaf vectors, which gives the stated lower bound.
\end{poc}

We can now compare the leading coefficients in the two cases. Set
\[
  \eta_d:=\min\left\{A_d,\frac{1}{d!(d-1)!}\right\}>0.
\]
If $r>0$, the second claim gives a leading coefficient at least $2A_d$.
If $r=0$, the second and third claims together add at least
$1/(d!(d-1)!)$ to the leading coefficient. Thus in both cases the leading
coefficient exceeds $A_d$ by at least $\eta_d$.

All error constants depend only on $d,r_0$: $L$ is fixed by these
parameters, $k\le d-1$, $r\le r_0$, and only finitely many cores and
remainder forests occur up to labelling. Together with
Proposition~\ref{prop:coefficient}, there is therefore a constant
$M_{d,r_0}>0$ such that
\[
  \tauiso(G)-\tauiso(K_{d,n-d})
  \ge \eta_dn^{d-1}-M_{d,r_0}n^{d-2}
  =n^{d-2}(\eta_dn-M_{d,r_0}).
\]
This is positive once $n>M_{d,r_0}/\eta_d$. Increasing the threshold to
cover the earlier counting requirements gives a single $n_0(d,r_0)$ for
all graphs in the proposition.
\end{proof}

\medskip

\section{Proof of the main theorem}\label{sec:mainproofs}

We now combine the stability theorem with the extremal comparison from the
previous section. A minimizer has no more types than the bipartite example,
so Theorem~\ref{thm:stability} applies; Proposition~\ref{prop:gap} then rules
out every bounded-remainder graph other than $K_{d,n-d}$.

\begin{proof}[Proof of Theorem~\ref{thm:main}]
Fix $d\ge3$ and let $n$ be sufficiently large. Let $G_n$ be any graph
minimizing $\tauiso$ among connected $n$-vertex graphs of minimum
degree at least $d$. Since $K_{d,n-d}$ belongs to this class,
Proposition~\ref{prop:coefficient} gives
\[
  \tauiso(G_n)\le\tauiso(K_{d,n-d})=O_d(n^{d-1}).
\]
Theorem~\ref{thm:stability} therefore gives a degree-$d$
common-neighbourhood class in $G_n$ with a remainder bounded in terms
of $d$. If $G_n\not\cong K_{d,n-d}$, Proposition~\ref{prop:gap} gives
\[
  \tauiso(G_n)>\tauiso(K_{d,n-d}),
\]
contradicting the choice of $G_n$. Thus every minimizer is isomorphic
to $K_{d,n-d}$, which proves both the lower bound and its equality
statement.
\end{proof}

\medskip

\section*{Acknowledgement}

This project was initiated at the IBS ECOPRO Combinatorics Summer School in
2026. We thank Xinbu Cheng for introducing this question to us, and we also thank Hong Liu for helpful discussions. Z.Y. was supported by the
Institute for Basic Science (IBS-R029-C4).

\medskip
\noindent\textbf{Statement of AI usage:}
Our initial argument gave the lower bound
$\Omega_d((n/(\log n)^{65})^{d-1})$ for connected graphs of minimum degree at least $d$. 
We had also proved the sharp lower bound $\Omega_d(n^{d-1})$ under
the additional assumption of $d$-connectivity.
We then used GPT-5.6 to carry out a more detailed analysis, which led to the improved lower bound $\Omega_d((n/(\log n)^2)^{d-1})$.
We then worked with GPT-6.0 and obtained the bound $\Omega_d(n^{d-1})$.
We subsequently reorganized and simplified this argument into the proof presented here.
The authors take full responsibility for the mathematical claims, proofs, and final presentation of the paper.

\end{document}